\documentclass{amsart} 
\usepackage{amssymb,latexsym,textcomp}

\usepackage[OT2,T1]{fontenc}

\theoremstyle{plain}
\newtheorem{theorem}{Theorem}[section]
\newtheorem{lemma}[theorem]{Lemma}
\newtheorem{corollary}[theorem]{Corollary}
\newtheorem{proposition}[theorem]{Proposition}

\theoremstyle{definition}

\newtheorem{problem}{Problem}[section]

\theoremstyle{remark}
\newtheorem{remark}{Remark}[section]

\DeclareMathOperator{\tr}{tr}%
\newcommand{\N}{\mathbf N}
\newcommand{\Z}{\mathbf Z}
\newcommand{\R}{\mathbf R}
\newcommand{\T}{\mathbf T}
\newcommand{\f}{\varphi}
\newcommand{\e}{\varepsilon}
\renewcommand{\d}{\mathrm d}
\newcommand{\skp}[2]{\left<#1,#2\right>}
\newcommand{\C}{\mathbf C}
\renewcommand{\a}{\alpha}
\newcommand{\luin}{\left|\!\left|\!\left|}
\newcommand{\ruin}{\right|\!\right|\!\right|}
\newcommand{\ovl}{\overline}
\newcommand{\Sc}{\mathcal C}

\numberwithin{equation}{section} 

\begin{document}
\title[Continuous generalization of Clarkson-McCarthy inequalities]{Continuous generalization of Clarkson-McCarthy inequalities} 

\author{Dragoljub J. Ke\v cki\' c}
\address{University of Belgrade\\ Faculty of Mathematics\\ Student\/ski trg 16-18\\ 11000 Beograd\\ Serbia}

\email{keckic@matf.bg.ac.rs}

\thanks{The author was supported in part by the Ministry of education and science, Republic of Serbia, Grant \#174034.}

\begin{abstract}
Let $G$ be a compact abelian group, let $\mu$ be the corresponding Haar measure, and let $\hat G$ be the Pontryagin dual of $G$. Further, let $C_p$ denote the Schatten class of operators on some separable infinite dimensional Hilbert space, and let $L^p(G;C_p)$ denote the corresponding Bochner space. If $G\ni\theta\mapsto A_\theta$ is the mapping belonging to $L^p(G;C_p)$ then,
$$\sum_{k\in\hat G}\left|\left|\int_G\ovl{k(\theta)}A_\theta\d\theta\right|\right|_p^p\le
    \int_G||A_\theta||_p^p\d\theta,\qquad p\ge2$$
$$\sum_{k\in\hat G}\left|\left|\int_G\ovl{k(\theta)}A_\theta\d\theta\right|\right|_p^p\le
    \left(\int_G||A_\theta||_p^q\d\theta\right)^{p/q},\qquad p\ge2.$$
$$\sum_{k\in\hat G}\left|\left|\int_G\ovl{k(\theta)}A_\theta\d\theta\right|\right|_p^q\le
    \left(\int_G||A_\theta||_p^p\d\theta\right)^{q/p},\qquad p\le2.$$

If $G$ is a finite group, the previous comprises several earlier obtained generalizations of Clarkson-McCarthy inequalities (e.g.\ $G=\Z_n$ or $G=\Z_2^n$), as well as the original inequalities, for $G=\Z_2$.

Other related inequalities are also obtained.
\end{abstract}




\maketitle

\section{Introduction}

Investigating uniformly convex spaces, Clarkson \cite{Clarkson} proved the following inequalities for $L^p$ norms:
\begin{equation}\label{Lp>2}
\begin{gathered}
2(\|f\|_p^p+\|g\|_p^p)\le\|f+g\|_p^p+\|f-g\|_p^p\le2^{p-1}(\|f\|_p^p+\|g\|_p^p),\qquad p\ge2\\
2(\|f\|_p^p+\|g\|_p^p)\ge\|f+g\|_p^p+\|f-g\|_p^p\ge2^{p-1}(\|f\|_p^p+\|g\|_p^p),\qquad p\le2
\end{gathered}
\end{equation}
\begin{equation}\label{Lp<2alt}
\|f+g\|_p^p+\|f-g\|_p^p\le2(\|f\|_p^q+\|g\|_p^q)^{p/q},\qquad p\ge 2,\; q=p/(p-1).
\end{equation}
\begin{equation}\label{Lp<2}
\|f+g\|_p^q+\|f-g\|_p^q\le2^{q-1}(\|f\|_p^p+\|g\|_p^p)^{q/p},\qquad p\le 2,\; q=p/(p-1).
\end{equation}

Later McCarthy \cite{McCarthy} generalized these inequalities to Schatten classes of operators. He replaced measurable functions $f$ and $g$ by compact operators $A$ and $B$, and $L^p$ norm by $C_p$ norm defined as
$$\|A\|_p=\left(\tr(|A|^p)\right)^{1/p}.$$
The inequalities he obtained was exactly the inequalities (\ref{Lp>2}), (\ref{Lp<2alt}) and (\ref{Lp<2}). In operator framework, these inequalities are usually referred as Clarkson-McCarthy inequalities. In what follows, we shall use the abbreviation CMC.

There were many generalization of CMC inequalities. Among others, Bhatia and Kittaneh \cite{BhatiaFudo}, proved the following inequalities for $n$-tuple of operators:
\begin{equation}\label{Cpn}
n\sum_{j=0}^{n-1}\|A_j\|_p^p\le\sum_{k=0}^{n-1}\Big\|\sum_{j=0}^{n-1}\omega_j^kA_j\Big\|_p^p\le n^{p-1}\sum_{j=0}^{n-1}\|A_j\|_p^p,
\end{equation}
for $p\ge2$, and corresponding reversed inequalities for $p\le2$, where $\omega_j=e^{2\pi i j/n}$ is the $j$-th degree of the $n$-th root of unity. They, also, proved a stronger inequality:
\begin{equation}\label{BKineq}
n^{-p/2}\luin\sum_{k=0}^{n-1}\left|\sum_{j=0}^{n-1}\omega_j^kA_j\right|^p\ruin\le
\luin\left(\sum_{j=0}^{n-1}|A_j|^2\right)^{p/2}\ruin\le \frac1n\luin\sum_{k=0}^{n-1}\left|\sum_{j=0}^{n-1}\omega_j^kA_j\right|^p\ruin,
\end{equation}
for all unitarily invariant norms $\luin\cdot\ruin$, and same complex numbers $\omega_j$.

After this work there were several further generalizations. Hirzallah and Kittaneh \cite{FudoIEOP} replaced $t\mapsto t^{p/2}$ by an arbitrary convex (concave) function $f$ and obtain
\begin{align}\label{Omer}\luin\sum_{k=0}^{n-1}f\left(\frac1n\Big|\sum_{j=0}^{n-1}\omega_j^kA_j\Big|^2\right)\ruin
\le&\luin f\left(\left(\sum_{j=0}^{n-1}|A_j|^2\right)^{1/2}\right)\ruin\le\\
\le&\frac1n\luin\sum_{k=0}^{n-1}f\left(\Big|\sum_{j=0}^{n-1}\omega_j^kA_j\Big|^2\right)\ruin,\notag
\end{align}
for any convex $f:[0,+\infty)\to[0,+\infty)$ with $f(0)=0$ and any unitarily invariant norm.

The aim of this paper is to generalize the preceding three inequalities, using the argument from \cite{FudoIEOP}, to the framework of compact abelian groups as it is stated in the abstract. The most technical part of the paper is Theorem \ref{ParsevalThm} which establishes the Parseval identity for operator valued abstract Fourier series.

The paper is organized as follows: in section 2, we quote known results concerning abstract harmonic analysis on compact abelian groups, unitaily invariant norms and Bochner integral. Also we derive some minor auxiliary statements. Section 3 is devoted to main results. Finally, in section 4, varying group $G$ we obtain corollaries. For instance, choosing $G=\Z_n$ we obtain (\ref{Cpn}), (\ref{BKineq}) and  (\ref{Omer}), choosing $G=\Z_2$ we obtain classic CMC inequalities, choosing $G=\Z_2^n$ we obtain a generalization of some results from \cite{Kato} and \cite{Japonians}, whereas for other choices of $G$ we get completely new results. Finally, in the last section we list problems that naturally arises from this work.

\section{Preliminaries}

\subsection*{Compact abelian groups} Let us recall some basic facts on abstract harmonic analysis on compact abelian groups. For more details the reader is referred to \cite{Edwards} or \cite{Folland}.

For any locally compact abelian topological group there is a left (and also right) invariant regular Borel measure $\mu$ which is unique up to multiplication by a positive scalar. This measure is known as \emph{Haar measure}. If $G$ is, moreover, compact then $\mu$ is finite and usually normalized such that $\mu(G)=1$.

Haar measure exists for nonabelian locally compact group, as well. In this case it is only left invariant. However, the further theory can not be applied to nonabelian groups. In what follows, $G$ will always be abelian.

A character on $G$ is a continuous homomorphism $k:G\to\T=\{z\in\C\;|\;|z|=1\}$. It is well known that the set of all characters on $G$ equipped with open-compact topology, denoted by $\hat G$, is also a topological group. The group $\hat G$ is called \emph{Pontryagin dual} of $G$. The topology on $\hat G$ is discrete if and only if $G$ is compact.

Throughout this paper $G$ will always denote the \emph{compact} abelian group. The elements of $G$ will be denoted by small Greek letters, $\theta$, $\varphi$, etc. Since its Pontryagin dual $\hat G$ is a discrete group, its elements will be denoted by $k$, $j$, $m$, $n$, etc, and integration with the corresponding Haar measure will be denoted by $\sum$-sign.

For a function $f\in L^1(G)$. the abstract Fourier coefficient of $f$ is given by
\begin{equation}\label{FTransform}
\hat f:\hat G\to\C,\qquad\hat f(k):=\int_G \ovl{k(\theta)}f(\theta)\d\mu(\theta).
\end{equation}
The abstract Fourier series of $f$ is
\begin{equation}\label{FSeries}
\sum_{k\in\hat G}\hat f(k)k(\theta),
\end{equation}
whereas, for finite $\Delta\subseteq\hat G$, the sum
$$S_\Delta(f)=\sum_{k\in\Delta}\hat f(k)k(\theta),$$
is called the partial sum of the expansion (\ref{FSeries}).

We shall use notation $\Delta\to\hat G$ for usual summation over arbitrary families. Namely, $S_\Delta(f)\to f$ (for instance) means that for all $\e>0$ there is a finite $\Delta_0\subseteq\hat G$ such that $d(S_\Delta(f),f)<\e$ for all $\Delta\supseteq\Delta_0$. Any convergence that appears in this note is a convergence within a certain metric space. Hence, we will always be able to choose a sequence $\Delta_n\subseteq\Delta_{n+1}$ such that $S_{\Delta_n}f\to f$.

We shall use the following facts:

\begin{proposition}\label{PropAHA} Let $f\in L^2(G)$. Then

\begin{enumerate}
\item\label{PropAHA1/2} For $k_1$, $k_2\in\hat G$ we have $\int_Gk_1(\theta)\ovl{k_2(\theta)}\d\mu(\theta)=0$ if $k_1\neq k_2$. If $k_1=k_2$ this integral is equal to $1$.

\item\label{PropAHA1} Fourier series converges to $f$ in $L^2$-norm, as $\Delta\to\hat G$, i.e.\ for any $\e>0$, there is a finite $\Delta_0\subseteq\hat G$ such that $\|S_\Delta(f)-f\|_{L^2(G)}<\e$, whenever $\Delta\supseteq\Delta_0$;

\item\label{PropAHA2} The Parseval identity holds: $||\hat f||_{L^2(\hat G)}=||f||_{L^2(G)}$, i.e.
\begin{equation}\label{Parseval}
\sum_{k\in\hat G}|\hat f(k)|^2=\sum_{k\in\hat G}\left|\int_G \ovl{k(\theta)}f(\theta)\d\mu(\theta)\right|^2=\int_G|f(\theta)|^2\d\theta.
\end{equation}
\end{enumerate}

(Notice that $L^2(G)\subseteq L^1(G)$, due to $\mu(G)<+\infty$.)
\end{proposition}

\begin{proof} (\ref{PropAHA1/2}) is \cite[Proposition 4.4]{Folland}, (\ref{PropAHA1}) and (\ref{PropAHA2}) is proved in \cite[Corollary 4.7]{Folland} (see also \cite[\S 2.7.2, \S2.7.3 and \S2.9.1]{Edwards}).
\end{proof}

\subsection*{Unitarily invariant norms}

Let $B(H)$ be the space of all bounded linear operators on a separable complex Hilbert space $H$. The absolute value of an operator $A\in B(H)$ is defined by $|A|=(A^*A)^{1/2}$, and the singular values of $A$, denoted by $s_j(A)$ is defined as eigenvalues of $|A|$ arranged in nonincreasing order counting multiplicity, i.e.\ $s_j(A)=\lambda_j(|A|)$.

A unitarily invariant norm, denoted by $\luin\cdot\ruin$, is a norm defined on a norm ideal $J_{\luin\cdot\ruin}$ in $B(H)$, satisfying the property that $\luin UAV\ruin=\luin A\ruin$ for all operators $A\in J_{\luin\cdot\ruin}$ and all unitary operators $U$, $V\in B(H)$. We also assume that $\luin\cdot\ruin$ is normalized, that is $\luin A\ruin=\|A\|$ for all rank one operators ($\|\cdot\|$ stands for usual operator norm). We shall abbreviate $J_{\luin\cdot\ruin}$ to $J$ if there is no risk of ambiguity. Each unitarily invariant norm $\luin\cdot\ruin$ is a symmetric gauge function of the singular values, and $J_{\luin\cdot\ruin}$ is a Banach space contained in the ideal of compact operators. The only exception are norms equivalent to the usual operator norm which are defined on the whole $B(H)$.

Among all unitarily invariant norms, the most examined are Schatten norms
\begin{equation}\label{SchattenNorms}
\|A\|_p=(\tr|A|^p)^{1/p}=\left(\sum_{j=1}^{+\infty}s^p_j(A)\right)^{1/p},
\end{equation}
where $\tr$ is the usual trace functional and $1\le p<+\infty$. The corresponding ideals will be denoted by $\Sc_p$. We retain definition of $\Sc_p$ for $0<p<1$. Though, for $p<1$, $\|\cdot\|_p$ defined by (\ref{SchattenNorms}) is not norm, but quasinorm. Some results in this paper are valid for $p<1$, e.g.\ Theorem \ref{TheoremPP} and its Corollaries.

The other example are Ky Fan norms
$$\|A\|_{(n)}=\sum_{j=1}^ns_j(A).$$
The importance of the latter is contained in part (\ref{UinProp3}) of the following Proposition where other basic properties of unitarily invariant norms are listed.

\begin{proposition}\label{UinProp} We have
\begin{enumerate}
\item\label{UinProp1/2} For any unitarily invariant norm $\luin\cdot\ruin$ we have $\|A\|\le\luin A\ruin\le\|A\|_1$, where $\|\cdot\|$ is usual operator norm, and $\|\cdot\|_1$ is Schatten $1$-norm.

\item\label{UinProp3} The inequality $\luin A\ruin\le\luin B\ruin$ holds for all unitarily invariant norms if and only if $\|A\|_{(n)}\le\|B\|_{(n)}$ for all $n$. This is known as \emph{Ky Fan dominance property}. The inequality $\luin A\ruin\le\luin B\ruin$ for all unitarily invariant norms should be understood as: If $B\in J_{\luin\cdot\ruin}$ then $A\in J_{\luin\cdot\ruin}$ and the inequality holds.

\item\label{UinProp2} The Ky Fan norms can be computed as
$$\|A\|_{(n)}=\sum_{j=1}^ns_j(A)=\max\sum_{j=1}^n|\skp{A\f_j}{\psi_j}|,$$
where $\max$ is taken over all orthonormal systems $\f_j$ and all orthonormal systems $\psi_j$. The maximum is attained if $|A|\f_j=s_j(A)\f_j$ and $\psi_j=U^*\f_j$ where $A=U|A|$ is the polar decomposition of $A$.

\item\label{UinProp1} If $0\le A\le B$ then $\sum_{j=1}^ns_j(A)\le\sum_{j=1}^ns_j(B)$ for all $n$ and therefore $\luin A\ruin\le\luin B\ruin$ for all unitarily invariant norms.
\end{enumerate}
\end{proposition}

\begin{proof} The proof of (\ref{UinProp1/2}) can be found in \cite[Chapter III, \S3]{GKrein}, of (\ref{UinProp3}) in \cite[Chapter III, \S4]{GKrein}, of (\ref{UinProp2}) in \cite[Chapter III, \S3]{GKrein}. Finally, (\ref{UinProp1}) is an immediate consequence of (\ref{UinProp2}).
\end{proof}

\begin{proposition}\label{SemiCont}
Let $A_n$ be an increasing sequence of positive operators from $J_{\luin\cdot\ruin}$, and let $A_n$ weakly converges to some $A\in J_{\luin\cdot\ruin}$. Then $\luin A_n\ruin\to\luin A\ruin$.
\end{proposition}

\begin{proof} By Proposition \ref{UinProp}-(\ref{UinProp1}), the sequence $\luin A_n\ruin$ is nondesreasing. Hence, the limit can be replaced by supremum. By the same argument, $\sup_n\luin A_n\ruin\le\luin A\ruin$.

The opposite inequality follows from the lower semicontinuity of $\luin\cdot\ruin$, that is $\liminf\luin A_n\ruin\le\luin A\ruin$, \cite[Theorem 2.7]{Simon} (see also \cite[Theorem III.5.2.]{GKrein} and \cite[Theorem 2.16]{Simon}).
\end{proof}

We shall deal with convex functions $\f:[0,+\infty)\to[0,+\infty)$. Note that such function must be nondecreasing. Although such a function is never operator monotone (that is, $A\le B$ does not imply $\f(A)\le\f(B)$), and not necessarily operator convex (that is, $\f(\lambda A+(1-\lambda)B)\le\lambda\f(A)+(1-\lambda)\f(B)$ need not be true in general), many scalar-valued inequalities can be extended to unitarily invariant norms.

\begin{lemma}\label{ConvexFunc} Let $\f:[0,+\infty)\to[0,+\infty)$ be a convex function with $\f(0)=0$.
\begin{enumerate}
\item\label{ConvexFunc1} If $\luin A\ruin\le\luin B\ruin$ in any unitarily invariant norm, then also $\luin \f(A)\ruin\le\luin\f(B)\ruin$ in any unitarily invariant norm. In particular, the conclusion follows for positive $A$ and $B$ such that $A\le B$.

\item\label{ConvexFunc2} If $A$, $B$ be any two positive operators, and $0\le\lambda\le1$ then $$\luin\f(\lambda A+(1-\lambda)B)\ruin\le\luin\lambda\f(A)+(1-\lambda)\f(B)\ruin$$ in any unitarily invariant norm.

\item\label{ConvexFunc3} If $A_n$ is a sequence of positive operators then $$\luin\sum_{n=1}^{+\infty}\f(A_n)\ruin\le\luin\f(\sum_{n=1}^{+\infty}A_n)\ruin$$ in any unitarily     invariant norm.

    The preceding inequality is reversed, if $\f:[0,+\infty)\to[0,+\infty)$ is a concave function, with $\f(0)=0$, $\f(+\infty)=+\infty$.
\end{enumerate}
\end{lemma}

\begin{proof} (\ref{ConvexFunc1}) Let $s_j(A)$ and $s_j(B)$ be the singular values of $A$ and $B$ respectively. By the assumption we have for all Ky Fan norms $\|A\|_{(n)}\le\|B\|_{(n)}$, i.e.
$$\sum_{j=1}^ns_j(A)\le\sum_{j=1}^ns_j(B).$$
Since $\f$ is convex and nondecreasing, by elementary Karamata inequality \cite{Karamata}, we have
$$\sum_{j=1}^n\f(s_j(A))\le\sum_{j=1}^n\f(s_j(B)),\qquad\mbox{i.e.}\qquad\|\f(A)\|_{(n)}\le\|\f(B)\|_{(n)}.$$
Now result follows from Ky Fan dominance property.

(\ref{ConvexFunc2}) This is \cite[Theorem 2.3]{Aujla}.

(\ref{ConvexFunc3}) For finite sums and convex $\f$, this was proved in \cite[Corollary 3.6]{Kosem} (see also \cite[Theorem 4.4]{Uchiyama} and \cite[\S 6.1]{AujlaLMA2012}). Let us prove for infinite sums. For any $n\in\N$ we have $A_1+\dots+A_n\le\sum_{n=1}^{+\infty}A_n$, which, by Proposition \ref{ConvexFunc}-(\ref{ConvexFunc1}) implies
$$\luin\f(A_1+\dots+A_n)\ruin\le\luin\f\left(\sum_{n=1}^{+\infty}A_n\right)\ruin.$$
Hence, by finite case
$$\luin\sum_{k=1}^n\f(A_k)\ruin\le\luin\f\left(\sum_{k=1}^nA_k\right)\ruin\le\luin\f\left(\sum_{n=1}^{+\infty}A_n\right)\ruin.$$
The result for convex $\f$, follows taking a limit $n\to+\infty$ according to Proposition \ref{SemiCont}.

Let $\f:[0,+\infty)\to[0,+\infty)$ be a concave function, with $\f(0)=0$, $\f(+\infty)=+\infty$. Then it has the inverse function $\f^{-1}:[0,+\infty)\to[0,+\infty)$ which is convex, with $\f^{-1}(0)=0$. Let $B_n=\f(A_n)$. By result for convex functions, we have
$$\luin\f^{-1}\left(\sum_{n=1}^{+\infty}B_n\right)\ruin\le\luin\sum_{k=1}^n\f^{-1}(B_n)\ruin,\quad\mbox{i.e.}\quad
    \luin\f^{-1}\left(\sum_{n=1}^{+\infty}\f(A_n)\right)\ruin\le\luin\sum_{k=1}^nA_n\ruin.$$
Apply Proposition \ref{ConvexFunc}-(\ref{ConvexFunc1}) to the previous inequality and $\f$ to obtain the conclusion.
\end{proof}

\subsection*{Bochner spaces} Let $(\Omega,\mu)$ be a measurable space and let $X$ be a Banach space. The Bochner space $L^p(\Omega;X)$ is defined as the set of strongly measurable functions $f:\Omega\to X$ such that
$$\|f\|_{L^p(\Omega;X)}:=\left(\int_\Omega\|f(t)\|_X^p\d\mu(t)\right)^{1/p}<+\infty,$$
after identification of $\mu$-almost everywhere equal functions. Here, strong measurability is equivalent to weak measurability (that is, the measurability of scalar functions $t\mapsto\Lambda(f(t))$ for all $\Lambda\in X^*$) and separability of the image of $f$.

Bochner integral is linear, additive with respect to disjoint union, and also there holds
\begin{equation}\label{BochnerLinear}
T\int_\Omega f(t)\d\mu(t)=\int_\Omega Tf(t)\d\mu(t)
\end{equation}
for all $f\in L^p(\Omega;X)$ and all bounded linear $T:X\to Y$.

Jensen inequality for unitarily invariant norms, Proposition \ref{ConvexFunc}-(\ref{ConvexFunc2}), can be extended to Bochner integral, using the same argument as in \cite[Theorem 2.3]{Aujla}.

\begin{proposition}\label{JensenIntegral} Let $\luin\cdot\ruin$ be some unitarily invariant norm, let $J$ be the corresponding ideal, and let $L^1(\Omega;J)$ be the Bochner space, where $(\Omega;\mu)$ is a measurable space such that $\mu(\Omega)=1$. For all $A\in L^1(\Omega;J)$ such that $A(t)\ge0$ for almost all $t\in\Omega$ (i.e.\ $A(t)$ is a positive operator), and all convex functions $\f:[0,+\infty)\to[0,+\infty)$ the following inequality holds:
\begin{equation}\label{JensenOperator}
\luin\f\left(\int_\Omega A(t)\d\mu(t)\right)\ruin\le\luin\int_\Omega\f(A(t))\d\mu(t)\ruin.
\end{equation}

If $\f:[0,+\infty)\to[0,+\infty)$ is concave with $\f(0)=0$, $\f(+\infty)=+\infty$ then the inequality (\ref{JensenOperator}) is reversed.
\end{proposition}

\begin{proof} Let $s_n$ be the eigenvalues of $\int_\Omega A(t)\d\mu(t)$ arranged in nonincreasing order counting possible multiplicities, and let $\xi_n$ be the corresponding unit eigenvectors. Then $\f(s_k)$ are eigenvalues of $\f\left(\int_\Omega A(t)\d\mu(t)\right)$ with respect to the same eigenvectors and
\begin{align*}\sum_{k=1}^n\f(s_k)=&\sum_{k=1}^n\f\left(\skp{\left(\int_\Omega A(t)\d\mu(t)\right)\xi_k}{\xi_k}\right)=\\
    =&\f\left(\int_\Omega\skp{A(t)\xi_k}{\xi_k}\d\mu(t)\right)\le\int_\Omega\f\left(\skp{A(t)\xi_k}{\xi_k}\right)\d\mu(t)
\end{align*}
by scalar Jensen inequality. Convexity of $\f$ also implies $\f(\skp{A\xi}{\xi})\le\skp{\f(A)\xi}\xi$ for any positive $A$ and any unit vector $\xi$. Therefore
$$\sum_{k=1}^n\f(s_k)\le\sum_{k=1}^n\int_\Omega\skp{\f(A(t))\xi_k}{\xi_k}\d\mu(t)=\sum_{k=1}^n\skp{\int_\Omega\f(A(t))\d\mu(t)\xi_k}{\xi_k}.$$
By Proposition \ref{UinProp}-(\ref{UinProp2}), we obtain
$$\sum_{k=1}^n\f(s_k)\le\sum_{k=1}^ns_k\left(\int_\Omega\f(A(t))\d\mu(t)\right),$$
i.e.
$$\left\|\f\left(\int_\Omega A(t)\d\mu(t)\right)\right\|_{(n)}\le\left\|\int_\Omega\f(A(t))\d\mu(t)\right\|_{(n)}$$
for all $n$. The result follows from Ky Fan dominance property (Proposition \ref{UinProp}-(\ref{UinProp3})).

If $\f$ is concave, then $\f^{-1}$ is convex, and applying the previously proved inequality to $\f^{-1}$ and $\f(A(t))$ we obtain
$$\luin\f^{-1}\left(\int_\Omega\f(A(t))\d\mu(t)\right)\ruin\le\luin\int_\Omega\f^{-1}(\f(A(t)))\d\mu(t)\ruin=
\luin\int_\Omega A(t)\d\mu(t)\ruin.$$
Apply Proposition \ref{ConvexFunc}-(\ref{ConvexFunc1}) to obtain the conclusion.
\end{proof}

\section{Main results}

First, we establish Parseval identity for functions from Bochner space, which is the key technical tool in this paper.

\begin{theorem}\label{ParsevalThm} Let $G$ be a compact abelian group, let $\hat G$ be its Pontryagin dual and let $\luin\cdot\ruin$ be a unitarily invariant norm on ideal $J$. For $A_\theta\in L^2(G;J)$, and $k\in\hat G$, the operators
$$B_k=\int_G\ovl{k(\theta)}A_\theta\d\mu(\theta)$$
are well defined and also
\begin{equation}\label{AbstractParseval}
\sum_{k\in\hat G}|B_k|^2=\int_G|A_\theta|^2\d\mu(\theta),
\end{equation}
where the series on the left hand side converges strongly.
\end{theorem}

\begin{proof} Since $\mu(G)<+\infty$, we have $L^2(G;J)\subseteq L^1(G;J)$. Therefore, by $|k(\theta)|=1$ we have
$$\int_G\luin\ovl{k(\theta)}A_\theta\ruin\d\mu(\theta)=\int_G\luin A_\theta\ruin\d\mu(\theta)<+\infty$$
and $B_k$ are well defined.

Next, for $\xi\in H$, and any mapping $\theta\mapsto X_\theta$ we have
$$\skp{\int_GX_\theta^*X_\theta\d\mu(\theta)\xi}\xi=\int_G\skp{X_\theta\xi}{X_\theta\xi}\d\mu(\theta)\ge0,$$
i.e.\ $\int_GX_\theta^*X_\theta\d\mu(\theta)\ge0$. (Here, we apply (\ref{BochnerLinear}) to $T:J\to\C$, $T(X)=\skp{X\xi}\xi$.) Hence
$$\int_G\left(A_\theta-\sum_{k\in\Delta}B_kk(\theta)\right)^*\left(A_\theta-\sum_{k\in\Delta}B_kk(\theta)\right)\d\mu(\theta)\ge0.$$
Expanding the left hand side, and taking into account Proposition \ref{PropAHA}-(\ref{PropAHA1/2}) we obtain
\begin{equation}\label{Bessel}
\sum_{k\in\Delta}B_k^*B_k\le\int_GA_\theta^*A_\theta\d\mu(\theta).
\end{equation}
(Alternatively, we can invoke the Bessel inequality for Hilbert $C^*$-modules to get (\ref{Bessel}).) In particular,
\begin{equation}\label{RedKonv}\sup_{\substack{\Delta\subseteq\hat G\\\Delta\mbox{ \scriptsize finite}}}\sum_{k\in\Delta}\skp{B_k^*B_k\xi}\xi<+\infty.
\end{equation}

We establish the first conclusion, that series in (\ref{AbstractParseval}) converges weakly, and even more strongly, due to its monotonicity.

Next, let us compute the difference between $\int|A_\theta|^2$ and the partial sum of $\sum_{k\in\hat G}|B_k|^2$.

For finite $\Delta\subseteq\hat G$, let $(S_\Delta A)_\theta=\sum_{k\in\Delta}B_kk(\theta)$ be the partial sum of the abstract Fourier series of $A_\theta$. Then, we have
\begin{equation}\label{Bessel2}
\sum_{k\in\Delta}B_k^*B_k=\sum_{k\in\Delta}\int_Gk(\theta)A_\theta^*B_k\d\mu(\theta)=\int_GA_\theta^*(S_\Delta A)_\theta\d\mu(\theta),
\end{equation}
once again invoking (\ref{BochnerLinear}), and hence
$$\int_G|A_\theta|^2\d\mu(\theta)-\sum_{k\in\Delta}|B_k|^2=\int_GA_\theta^*(A_\theta-(S_\Delta A)_\theta)\d\mu(\theta).$$
Choose unit vectors $\xi$, $\eta\in H$, to get
\begin{equation}\label{Equality1}
\skp{\left(\int_G|A_\theta|^2\d\mu(\theta)-\sum_{k\in\Delta}|B_k|^2\right)\xi}\eta=\int_G\skp{(A_\theta-(S_\Delta A)_\theta)\xi}{A_\theta\eta}\d\mu(\theta).
\end{equation}
We shall prove that the right hand side of (\ref{Equality1}) tends to $0$ for suitable sequence of finite $\Delta_n\subseteq\hat G$.

We have
\begin{align*}\skp{(S_\Delta A)_\theta\xi}\eta=&\sum_{k\in\Delta}\skp{B_k\xi}\eta k(\theta)=\\
=&\sum_{k\in\Delta}\int_G\ovl{k(\f)}\skp{A_\f\xi}\eta\d\mu(\f)k(\theta)=\\
=&(S_\Delta\skp{A\xi}\eta)_\theta.
\end{align*}
However, $|\skp{A_\theta\xi}\eta|\le\|A_\theta\|\,\|\xi\|\,\|\eta\|\le\luin A_\theta\ruin\,\|\xi\|\,\|\eta\|\in L^2(G)$. Therefore, by Proposition \ref{PropAHA}-(\ref{PropAHA1}), $\skp{(S_\Delta A)_\theta\xi}\eta$ converges in $L^2$ norm to $\skp{A_\theta\xi}\eta$. Moreover, we can choose a sequence of finite sets $\Delta_n\subseteq\Delta_{n+1}$ such that
\begin{equation}\label{SDeltan}
\skp{(S_{\Delta_n} A)_\theta\xi}\eta\to\skp{A_\theta\xi}\eta
\end{equation}
for almost all $\theta\in G$.

To show that we can pass the limit on the right hand side of (\ref{Equality1}) to the integrand, let us show that the family of functions $\theta\mapsto\skp{(A_\theta-(S_\Delta A)_\theta)\xi}{A_\theta\eta}$ is uniformly integrable.

Let $\e>0$ be arbitrary. By (\ref{RedKonv}), there is a finite set $\Delta_0\subseteq\hat G$ such that
\begin{equation}\label{Ostatak}
\sum_{k\notin\Delta_0}\skp{B_k^*B_k\xi}\xi<\frac{\e^2}{4M},
\end{equation}
where
$$M=\int_G\|A_\theta\|^2\d\mu(\theta)\le\int_G\luin A_\theta\ruin^2\d\mu(\theta)<+\infty.$$
Let $p$ be the cardinality of $\Delta_0$. The function $\theta\mapsto\|A_\theta\|^2+pM^{1/2}\|A_\theta\|$ is integrable. Therefore, there is $\delta>0$ such that
\begin{equation}\label{IntegralOcena}
\int_E(\|A_\theta\|^2+pM^{1/2}\|A_\theta\|)\d\mu(\theta)<\frac\e2,
\end{equation}
for all $E\subseteq G$ such that $\mu(E)<\delta$.

For such $\Delta_0$ and such $E$, we have
\begin{align*}&\left|\int_E\skp{(A_\theta-(S_\Delta A)_\theta)\xi}{A_\theta\eta}\d\mu(\theta)\right|\le\int_E\left|\skp{\left(A_\theta-\sum_{k\in\Delta}B_kk(\theta)\right)\xi}{A_\theta\eta}\right|\d\mu(\theta)\le\\
&\qquad\le\int_E\left|\skp{\left(A_\theta-\sum_{k\in\Delta\cap\Delta_0}B_kk(\theta)\right)\xi}{A_\theta\eta}\right|\d\mu(\theta)+\\
&\qquad\qquad+\int_E\left|\skp{\sum_{k\in\Delta\setminus\Delta_0}B_kk(\theta)\xi}{A_\theta\eta}\right|\d\mu(\theta)=\\
&\qquad=S_1+S_2.
\end{align*}

Since
$$\|B_k\|\le\int_G\|A_\theta\ovl{k(\theta)}\|\d\mu(\theta)\le\left(\int_G\luin A_\theta\ruin^2\d\mu(\theta)\right)\le M^{1/2}$$
we can estimate the first summand $S_1$ as
\begin{align*}
S_1\le&\int_E\left(\sum_{k\in\Delta\cap\Delta_0}\left|\skp{A_\theta\xi}{A_\theta\eta}\right|+\left|\skp{B_kk(\theta)\xi}{A_\theta\eta}\right|\right)\d\mu(\theta)\le\\
\le&\int_E\left(\|A_\theta\xi\|\|A_\theta\eta\|+\sum_{k\in\Delta_0}\|B_k\xi\|\|A_\theta\eta\|\right)\d\mu(\theta)\le\\
\le&\int_E\left(\|A_\theta\|^2+pM^{1/2}\|A_\theta\|\right)\d\mu(\theta)<\frac\e2
\end{align*}
for unit vectors $\xi$, $\eta\in H$ by (\ref{IntegralOcena})

Let us estimate the second summand $S_2$. We have
\begin{align*}
S_2\le&\int_E\left|\skp{\sum_{k\in\Delta\setminus\Delta_0}B_kk(\theta)\xi}{A_\theta\eta}\right|\d\mu(\theta)\le\\
\le&\int_G\left\|\sum_{k\in\Delta\setminus\Delta_0}B_kk(\theta)\xi\right\|\|A_\theta\|\d\mu(\theta)\le\\
\le&\left(\int_G\left\|\sum_{k\in\Delta\setminus\Delta_0}B_kk(\theta)\xi\right\|^2\d\mu(\theta)\right)^{1/2}\left(\int_G\|A_\theta\eta\|^2\d\mu(\theta)\right)^{1/2}\le\\
\le&M^{1/2}\left(\int_G\skp{\sum_{k\in\Delta\setminus\Delta_0}B_kk(\theta)\xi}{\sum_{j\in\Delta\setminus\Delta_0}B_jj(\theta)\xi}\d\mu(\theta)\right)^{1/2}=\\
=&M^{1/2}\left(\int_G\sum_{k,j\in\Delta\setminus\Delta_0}\skp{B_j^*B_k\xi}\xi k(\theta)\ovl{j(\theta)}\d\mu(\theta)\right)^{1/2}\le\\
\le&M^{1/2}\left(\sum_{k\notin\Delta_0}\skp{B_k^*B_k\xi}\xi\right)^{1/2}<M^{1/2}\left(\frac{\e^2}{4M}\right)^{1/2}=\frac\e2,
\end{align*}
where we use Proposition \ref{PropAHA}-(\ref{PropAHA1/2}), and (\ref{Ostatak})

Thus
$$\left|\int_E\skp{(A_\theta-(S_\Delta A)_\theta)\xi}{A_\theta\eta}\d\mu(\theta)\right|<\e,$$
whenever $\mu(E)<\delta$, with $\delta$ does not depend on $\Delta$.

This ensures that $\theta\mapsto\skp{(A_\theta-(S_\Delta A)_\theta)\xi}{A_\theta\eta}$ is uniformly integrable, and by Vitali convergence theorem, we can pass the limit as $\Delta\to\hat G$ to the integrand at the right hand side of (\ref{Equality1}). By (\ref{SDeltan}), we obtain (\ref{AbstractParseval}), where the series converges in the weak operator topology. The entries of the sum are positive, hence partial sums are increasing. Therefore, the convergence is, moreover, strong.
\end{proof}

\begin{remark}\label{HilbertModule} We have $L^2(G;J)\subseteq L^2(G;B(H))$. Carefully reading the proof, we see that we use only the following properties of $B(H)$: ($i$) it is a $C^*$ algebra; ($ii$) it is closed under weak and strong limits; ($iii$) it has a unit. Therefore, equality (\ref{AbstractParseval}) holds for $A_\theta\in L^2(G;\mathcal A)$, where $\mathcal A$ is an arbitrary $W^*$-algebra.

Also, $L^2(G;\mathcal A)$ can be regarded as a Hilbert $W^*$-module with right multiplication $A_\theta\cdot X=A_\theta X$ and the $\mathcal A$-valued inner product $\skp{A_\theta}{B_\theta}=\int_G A_\theta^*B_\theta\d\mu(\theta)$. The idea to prove (\ref{AbstractParseval}) by showing that $\{k(\theta)\cdot I\;|\;k\in\hat G\}^\perp=\{0\}$, where $I$ is the unit of $\mathcal A$ is misleading. Namely, there are examples of subspaces of some Hilbert module that have trivial orthogonal complement and are not dense.
\end{remark}

Using operator valued Parseval identity, we are able to derive continuous counterparts of CMC inequalities.

\begin{theorem}\label{TheoremUin} Let $\f:[0,+\infty)\to[0,+\infty)$ be a convex function such that $\f(0)=0$ and let $A_\theta$, $\theta\in G$ be such that $\f(|A_\theta|^2)\in L^1(G;J)$, $J=J_{\luin\cdot\ruin}$ arbitrary. Then it holds:
\begin{equation}\label{Inequality1}\luin\sum_{k\in\hat G}\f\left(\left|\int_G\ovl{k(\theta)}A_\theta\d\theta\right|^2\right)\ruin\le
    \luin \int_G\f(|A_\theta|^2)\d\theta\ruin.
\end{equation}

If $\f:[0,+\infty)\to[0,+\infty)$ is a concave function, $\f(0)=0$, $\f(+\infty)=+\infty$, then the inequality is reversed.
\end{theorem}

\begin{proof}By (\ref{AbstractParseval}) and (\ref{JensenOperator}) we have
\begin{equation}\label{Proof1}\luin\f\left(\sum_{k\in\hat G}\left|\int_G\ovl{k(\theta)}A_\theta\d\theta\right|^2\right)\ruin=
\luin\f\left(\int_G|A_\theta|^2\d\theta\right)\ruin\le\luin\int_G\f(|A_\theta|^2)\d\theta\ruin.
\end{equation}
Also by (\ref{AbstractParseval}), the operator $\f\left(\sum_{k\in\hat G}\left|\int_GA_\theta\d\theta\right|^2\right)$ belongs to $J$.

By Proposition \ref{ConvexFunc}-(\ref{ConvexFunc3}) we have
\begin{equation}\label{Proof3}\luin\sum_{k\in\hat G}\f\left(\left|\int_G\ovl{k(\theta)}A_\theta\d\theta\right|^2\right)\ruin\le
    \luin\f\left(\sum_{k\in\hat G}\left|\int_G\ovl{k(\theta)}A_\theta\d\theta\right|^2\right)\ruin.
\end{equation}

The conclusion (\ref{Inequality1}) follows from (\ref{Proof1}) and (\ref{Proof3}).

If $\f$ is concave, note that the inequality in (\ref{Proof1}) is reversed due to Proposition \ref{JensenIntegral}, and in (\ref{Proof3}) by Proposition \ref{ConvexFunc}-(\ref{ConvexFunc3}).
\end{proof}

\begin{theorem}\label{TheoremPP} Let $1\le p<+\infty$, and let $A_\theta\in L^p(G;\Sc_p)$. Then for $p\ge2$ we have
\begin{equation}\label{Inequality2}\sum_{k\in\hat G}\left\|\int_G\ovl{k(\theta)}A_\theta\d\theta\right\|_p^p\le
    \int_G\|A_\theta\|_p^p\d\theta,
\end{equation}
whereas for $0<p\le 2$ the inequality is reversed.
\end{theorem}

\begin{proof} Put $\f(t)=t^{p/2}$, $p\ge2$, which is a convex function, for $p\ge2$ and $\luin A\ruin=||A||_1=\tr(|A|)$. We obtain
$$\tr\left(\sum_{k\in\hat G}\left|\int_G\ovl{k(\theta)}A_\theta\d\theta\right|^p\right)\le
    \tr\left(\int_G|A_\theta|^p\d\theta\right),$$
which leads to (\ref{Inequality2}), since $\tr$ is a bounded linear functional.

For $0<p\le 2$, the function $\f(t)=t^{p/2}$ is concave, $\f(0)=0$ and $\f(+\infty)=+\infty$, so by the same argument we obtain the reversed inequality.
\end{proof}

Formula (\ref{Inequality2}) is a generalization of right inequality in (\ref{Cpn}), which we shall prove in the next section. Concerning left inequality in (\ref{Cpn}), it follows from the right inequality by substitution $B_n=\sum_{j=0}^{n-1}\omega_j^kA_j$. This is possible due to the fact that $\Z_n$ is selfdual in Pontryagin sense. Nothing similar can be said for general compact abelian group $G$. It need not be isomorphic to its dual group $\hat G$ in general.

Nevertheless, a partial substitution for left inequality in (\ref{Cpn}) might be the following:

\begin{theorem}\label{TheoremAlpha} Let $\f:[0,+\infty)\to[0,+\infty)$ be a convex function $\f(0)=0$ and let $\a_k$, $k\in\hat G$ be a family of positive reals such that
\begin{equation}\label{Suma1}
\sum_{k\in\hat G}\a_k=1.
\end{equation}
If $A_\theta\in L^1(G;\Sc_p)$, for some $p\ge2$, then
\begin{equation}\label{Inequality3}\left\|\int_G|A_\theta|\d\theta\right\|_p^p\le
    \sum_{k\in\hat G}\alpha_k^{1-p/2}\left\|\int_G\ovl{k(\theta)}A_\theta\d\theta\right\|_p^p,
\end{equation}
provided that the term on the right hand side is finite.
\end{theorem}

\begin{proof} For $\Delta\subseteq\hat G$ let $\nu(\Delta)=\sum_{k\in\Delta}\a_k$. Then $\nu$ is a measure with $\nu(\hat G)=1$. By (\ref{JensenOperator})
$$\luin\f\left(\sum_{k\in\hat G}\left|\int_G\ovl{k(\theta)}A_\theta\d\theta\right|^2\right)\ruin\le
    \luin\sum_{k\in\hat G}\alpha_k\f\left(\frac1{\alpha_k}\left|\int_G\ovl{k(\theta)}A_\theta\d\theta\right|^2\right)\ruin.$$
and by (\ref{AbstractParseval}) we get
$$\luin\f\left(\int_G|A_\theta|^2\d\theta\right)\ruin\le\luin\sum_{k\in\hat G}\alpha_k\f\left(\frac1{\alpha_k}\left|\int_G\ovl{k(\theta)}A_\theta\d\theta\right|^2\right)\ruin.$$

Once again, take $f(t)=t^{p/2}$, $p\ge2$, and $\luin A\ruin=\|A\|_1=\tr(|A|)$. We obtain
$$\tr\left(\int_G|A_\theta|^2\d\theta\right)^{p/2}\le\sum_{k\in\hat G}\alpha_k\tr\left(\frac1{\alpha_k}\left|\int_G\ovl{k(\theta)}A_\theta\d\theta\right|^2\right)^{p/2}.
$$
However, by convexity of $t\mapsto t^2$ we have
$$\tr\left(\int_G|A_\theta|\d\theta\right)^p\le
    \tr\left(\int_G|A_\theta|^2\d\theta\right)^{p/2}$$
and (\ref{Inequality3}) follows.
\end{proof}

In two next Theorems, we prove the counterpart of Clarkson inequalities (\ref{Lp<2alt}) and (\ref{Lp<2}) using complex interpolation, which is a standard procedure. These results hold only for $p\ge1$/

\begin{theorem}\label{TheoremPQ}
For all $1\le p\le 2$, and $A_\theta\in L^p(G;\Sc_p)$ there holds
\begin{equation}\label{Inequalityp<2}
\sum_{k\in\hat G}\left\|\int_G\ovl{k(\theta)}A_\theta\d\mu(\theta)\right\|_p^q\le
    \left(\int_G\|A_\theta\|_p^p\right)^{q/p},
\end{equation}
where $q$ is conjugate to $p$, i.e.\ $q=p/(p-1)$.
\end{theorem}

\begin{proof} The proof can be obtained using complex interpolation as it was done in \cite{FackKosaki} and then repeated in \cite{BhatiaFudo}. Therefore only the outline will be given.

First, prove the following inequality
\begin{equation}\label{TraceIneq}
\left|\tr\sum_{k\in\Delta}Y_kB_k\right|\le\left(\sum_{k\in\Delta}\|Y_k\|_q^p\right)^{1/p}
    \left(\int_G\|A_\theta\|_p^p\d\mu(\theta)\right)^{1/p},
\end{equation}
where $Y_k\in \Sc_q$, $B_k=\int_G\ovl{k(\theta)}A_\theta\d\mu(\theta)$, and $\Delta\subseteq\hat G$ is finite. (We choose a finite subset of $\hat G$ to avoid complications with convergence till the end of the proof.)

Indeed, consider the function $f(z)$ defined for $1/2\le\Re z\le1$ by
$$f(z)=\tr\sum_{k\in\Delta}Y_k(z)B_k(z),$$
where
$$Y_k(z)=\|Y_k\|_q^{(p+q)z-q}V_k|Y_k|^{q-qz},\quad Y_k=V_k|Y_k|,$$
$$B_k(z)=\int_G\ovl{k(\theta)}A_\theta(z)\d\mu(\theta),\quad A_\theta(z)=|A_\theta|^{pz}W_\theta,\quad A_\theta=|A_\theta|W_\theta.$$
Then estimate
$$|\tr(Y_k(1+it)A_\theta(1+it))|\le\|Y_k\|_q^p\|A_\theta\|_p^p,$$
and hence
\begin{equation}\label{Estimate1}
|f(1+it)|\le\left(\sum_{k\in\Delta}\|Y_k\|_q^p\right)\left(\int_G||A_\theta||_p^p\d\mu(\theta)\right).
\end{equation}
Also,
\begin{align*}|f(1/2+it)|\le&\sum_{k\in\Delta}||Y_k(1/2+it)||_2||B_k(1/2+it)||_2\le\\
    \le&\left(\sum_{k\in\Delta}||Y_k(1/2+it)||_2^2\right)^{1/2}\left(\sum_{k\in\Delta}||B_k(1/2+it)||_2^2\right)^{1/2}.
\end{align*}
However, by (\ref{Inequality2}) with $p=2$ we obtain
$$\sum_{k\in\Delta}||B_k(1/2+it)||_2^2\le\int_G||A_\theta(1/2+it)||_2^2\d\mu(\theta)=\int_G||A_\theta||_p^p\d\mu(\theta),$$
which yields (note $||Y_k(1/2+it)||_2=||Y_k||_q^{p/2}$):
\begin{equation}\label{Estimate2}|f(1/2+it)|\le\left(\sum_{k\in\Delta}||Y_k||_q^p\right)^{1/2}\left(\int_G||A_\theta||_p^p\d\mu(\theta)\right)^{1/2}
\end{equation}

From (\ref{Estimate1}) and (\ref{Estimate2}) we get (\ref{TraceIneq}) using three line Theorem \cite[Chapter III, \S13]{GKrein} or \cite[Theorem 2.9]{Simon}, since the left hand side of (\ref{TraceIneq}) is equal to $f(1/p)$.

Once (\ref{TraceIneq}) is proved, set $Y_k=||B_k||_p^{q-p}|B_k|^{p-1}U_k^*$, where $B_k=U_k|B_k|$ and the conclusion follows, by passing to the limit $\Delta\to\hat G$.
\end{proof}

\begin{theorem}\label{TheoremPQ2} For all $p\ge 2$, and $A_\theta\in L^q(G;\Sc_p)$ there holds
\begin{equation}\label{Inequalityp>2}
\sum_{k\in\hat G}\left\|\int_G\ovl{k(\theta)}A_\theta\d\mu(\theta)\right\|_p^p\le
    \left(\int_G\|A_\theta\|_p^q\right)^{p/q},
\end{equation}
where $q$ is conjugate to $p$, i.e.\ $q=p/(p-1)$.
\end{theorem}

\begin{proof} The proof is very similar to the proof of previous Theorem. Therefore, we only give differences. First we prove the inequality
\begin{equation}\label{TraceIneqP>2}
\left|\tr\sum_{k\in\Delta}Y_kB_k\right|\le\left(\sum_{k\in\Delta}\|Y_k\|_q^q\right)^{1/q}
    \left(\int_G\|A_\theta\|_p^q\d\mu(\theta)\right)^{1/q},
\end{equation}
where $Y_k\in Sc_q$, $B_k=\int_G\ovl{k(\theta)}A_\theta\d\mu(\theta)$, and $\Delta\subseteq\hat G$ is finite.

Consider the function $f(z)$ (slightly different from those from Theorem \ref{TheoremPQ}) defined in the same strip $1/2\le\Re z\le1$ by
$$f(z)=\tr\sum_{k\in\Delta}Y_k(z)B_k(z),$$
where
$$Y_k(z)=V_k|Y_k|^{qz},\quad Y_k=V_k|Y_k|,$$
$$B_k(z)=\int_G\ovl{k(\theta)}A_\theta(z)\d\mu(\theta),\quad A_\theta(z)=\|A_\theta\|_p^{(p+q)z-p}|A_\theta|^{p-pz}W_\theta,\quad A_\theta=|A_\theta|W_\theta.$$
Then use the same estimates for $z=1+it$ and $z=1/2+it$, and finally use three line theorem for $z=1/q\in(1/2,1)$.

Once (\ref{TraceIneqP>2}) is proved, set $Y_k=|B_k|^{p-1}U_k^*$, where $B_k=U_k|B_k|$ and the conclusion follows, by passing to the limit $\Delta\to\hat G$.
\end{proof}

\section{Corollaries}

Varying the group $G$, from Theorems \ref{TheoremUin}, \ref{TheoremPP}, \ref{TheoremAlpha}, \ref{TheoremPQ} and \ref{TheoremPQ2}, we obtain different earlier published results, as well as some new results.

For $G=\Z_n$ we get results from \cite{BhatiaFudo} and \cite{FudoIEOP}.

\begin{corollary}Let $n\in\N$, let $\omega_j=e^{2\pi i j/n}$, and let $A_j\in\Sc_p$, $j=1,2,\dots,n$. Then
\begin{equation}\label{ZnPP}
n\sum_{j=0}^{n-1}\|A_j\|_p^p\le\sum_{k=0}^{n-1}\Big\|\sum_{j=0}^{n-1}\omega_j^kA_j\Big\|_p^p\le n^{p-1}\sum_{j=0}^{n-1}\|A_j\|_p^p,
\end{equation}
for $p\ge 2$. For $0<p\le2$ the inequalities are reversed.
\begin{equation}\label{ZnPQ}
\sum_{k=0}^{n-1}\Big\|\sum_{j=0}^{n-1}\omega_j^kA_j\Big\|_p^q\le n\left(\sum_{j=0}^{n-1}\|A_j\|_p^p\right)^{q/p},
\end{equation}
for $1\le p\le 2$, $q=p/(p-1)$, and
\begin{equation}\label{ZnPQ2}
\sum_{k=0}^{n-1}\Big\|\sum_{j=0}^{n-1}\omega_j^kA_j\Big\|_p^p\le n\left(\sum_{j=0}^{n-1}\|A_j\|_p^q\right)^{p/q},
\end{equation}
for $p\ge2$ and $q=p/(p-1)$.
\end{corollary}

\begin{proof}Consider $G=\Z^n$. Its Haar measure is the counting measure divided by $n$, and Pontryagin dual is also $\Z_n$. Indeed, since $\Z_n=\{1,a,\dots,a^{n-1}\}$ for some generator $a$, any homomorphism $k:\Z_n\to\T$ is determined, by $k(a)$. From $a^n=1$ we deduce $k(a)^n=1$. Hence $k(a)=\omega_j$ for some $j=0,1,\dots,n-1$. Then $k(a^l)=\omega_j^l$. Hence (\ref{Inequality2}), (\ref{Inequalityp<2}) and (\ref{Inequalityp>2}) are reduced to
$$\sum_{k=0}^{n-1}\Big\|\frac1n\sum_{j=0}^{n-1}\ovl\omega_j^kA_j\Big\|_p^p\le\frac1n\sum_{j=0}^{n-1}\|A_j\|_p^p,$$
$$\sum_{k=0}^{n-1}\Big\|\frac1n\sum_{j=0}^{n-1}\ovl\omega_j^kA_j\Big\|_p^q\le\left(\frac1n\sum_{j=0}^{n-1}\|A_j\|_p^p\right)^{q/p},$$
$$\sum_{k=0}^{n-1}\Big\|\frac1n\sum_{j=0}^{n-1}\ovl\omega_j^kA_j\Big\|_p^p\le\left(\frac1n\sum_{j=0}^{n-1}\|A_j\|_p^q\right)^{p/q}.$$
These inequalities are equivalent to right inequality in (\ref{ZnPP}), (\ref{ZnPQ}) and (\ref{ZnPQ2}). Indeed, complex conjugation is the automorphism of the group $\{1,\omega_j,\omega_j^2,\dots,\omega_j^{n-1}\}$ and it only affects the corresponding sums by permutation of $A_j$.

Left inequality in (\ref{ZnPP}) can be obtained either from the right inequality by substitutions $B_k=\sum_{j=0}^{n-1}\omega_j^kA_j$ or by Theorem \ref{TheoremAlpha}, choosing $\a_k=1/k$.
\end{proof}

\begin{remark}Inequalities (\ref{ZnPP}) and (\ref{ZnPQ}) were proved in \cite{BhatiaFudo} and \cite{FudoIEOP}, as well as the inequality (\ref{Omer}) which is the consequence of (\ref{Inequality1}).
\end{remark}

\begin{remark} Notice that all constants that appears in any CMC inequality become $1$, if we normalize Haar measure.
\end{remark}

\begin{remark}As a special case of the preceding Corollary, if $G=\Z_2$ we obtain original CMC inequalities (\ref{Lp>2}), (\ref{Lp<2alt}) and (\ref{Lp<2}).
\end{remark}

For $G=\Z_2^n$ we get the following result concerning Littlewood matrices. They are defined inductively as
\begin{equation}\label{Little}
L_1=\left[\begin{matrix}1&1\\1&-1\end{matrix}\right],\qquad L_{n+1}=\left[\begin{matrix}L_n&L_n\\L_n&-L_n\end{matrix}\right].
\end{equation}

\begin{corollary}Let $n\in\N$, let $A_j\in \Sc_p$ for $1\le j\le 2^n$ and let $\e_{ij}$ be the entries of the Littlewood matrix $L_n$. Then
\begin{equation}\label{Littlewood1}
\sum_{i=1}^{2^n}\left\|\sum_{j=1}^{2^n}\e_{ij}A_j\right\|_p^p\le 2^{n(p-1)}\sum_{j=1}^{2^n}\|A_j\|_p^p,\qquad p\ge 2,
\end{equation}
\begin{equation}\label{Littlewood2}
\sum_{i=1}^{2^n}\left\|\sum_{j=1}^{2^n}\e_{ij}A_j\right\|_p^q\le
2^{n}\left(\sum_{j=1}^{2^n}\|A_j\|_p^p\right)^{q/p},\qquad 1\le p\le2,q=\frac p{p-1}.
\end{equation}
\begin{equation}\label{Littlewood3}
\sum_{i=1}^{2^n}\left\|\sum_{j=1}^{2^n}\e_{ij}A_j\right\|_p^p\le
2^{n}\left(\sum_{j=1}^{2^n}\|A_j\|_p^q\right)^{p/q},\qquad p\ge2,q=\frac p{p-1}.
\end{equation}
\end{corollary}

\begin{proof} Consider the group $\Z_2^n$. It has $n$ generators, say $a_1$, $a_2$, $\dots$, $a_n$, all of them of order $2$. Therefore
$$\Z_2^n=\{a_1^{m_1}a_2^{m_2}\dots a_n^{m_n}\;|\;(m_1,\dots,m_n)\in\{0,1\}^n\}.$$
Any character $k:\Z_2^n\to\T$ is determined by $k(a_j)$, $1\le j\le n$. It has to be $k(a_j)=\pm1$, since $k(a_j)^2=k(a_j^2)=k(1)=1$. Therefore, there is $2^n$ distinct characters on $\Z_2^n$.

Let us show that the rows of the Littlewood matrix are exactly the images $k(a)$, $a\in\Z_2^n$ in the lexicographic order, namely in the order
$$a_1^{m_1}a_2^{m_2}\dots a_n^{m_n}\le a_1^{m'_1}a_2^{m'_2}\dots a_n^{m'_n}\quad\mbox{iff}\quad m_1<m'_1\quad\mbox{or}\quad m_1=m'_1\wedge m_2<m'_2,\dots.$$

For $n=1$, $\Z_2^1=\{a_1^0,a_1^1\}$ and there are exactly two characters $k_1\equiv1$ and $k_2(a_1^0)=1$, $k_2(a_1^1)=-1$. This corresponds to the rows of $L_1$ (see formula (\ref{Little}).

Let the statement be true for some $n\in\N$. Then the first $2^n$ elements of $\Z_2^{n+1}$ are $a_1^0a_2^{m_2}\dots a_{n+1}^{m_{n+1}}$, whereas other $2^n$ elements are $a_1^1a_2^{m_2}\dots a_{n+1}^{m_{n+1}}$. Divide characters on $\Z_2^{n+1}$ into two slots. Let the first consist of those $k$ for which $k(a_1)=1$, and the second of $k$ for which $k(a_1)=-1$.

For $k\in\hat\Z_2^{n+1}$ in the first slot, there is a unique $k'\in\hat\Z_2^n$ such that
$$k(a_1^0a_2^{m_2}\dots a_{n+1}^{m_{n+1}})=k(a_1^1a_2^{m_2}\dots a_{n+1}^{m_{n+1}})=k'(a_2^{m_2}\dots a_{n+1}^{m_{n+1}}).$$
If $i$th row of $L_n$ corresponds to $k'$, then two copies of this row corresponds to $k$, and these two copies make exactly the $i$th row of $L_{n+1}$.

For $k\in\hat\Z_2^{n+1}$ in the second slot, there is a unique $k'\in\hat\Z_2^n$ such that
$$k(a_1^0a_2^{m_2}\dots a_{n+1}^{m_{n+1}})=k'(a_2^{m_2}\dots a_{n+1}^{m_{n+1}}),\quad k(a_1^1a_2^{m_2}\dots a_{n+1}^{m_{n+1}})=-k'(a_2^{m_2}\dots a_{n+1}^{m_{n+1}}).$$
If $i$th row of $L_n$ corresponds to $k'$, then two copies of this row with the second copy multiplied by $-1$ corresponds to $k$, and these two copies make exactly the $(2^n+i)$th row of $L_{n+1}$.

Thus, it is proved that $\e_{ij}$ are values of $i$th character in $\Z_2^n$, i.e.\ $\e_{ij}=k_i(b_j)$, where $\Z_2^n=\{b_1,\dots,b_{2^n}\}$, $\hat\Z_2^n=\{k_1,\dots,k_{2^n}\}$. Hence
$$\sum_{j=1}^{2^n}\e_{ij}A_j=\sum_{j=1}^{2^n}k_i(b_j)A_j=2^n\int_{\Z^n}k_i(b)A_b\d\mu(b),$$
since Haar measure on $Z_2^n$ is the counting measure, divided by $2^n$. Therefore, (\ref{Littlewood1}) becomes
$$\sum_{k\in\hat\Z_2^n}\left\|2^n\int_{\Z_2^n}k_i(b)A_b\d\mu(b)\right\|_p^p\le2^{n(p-1)} 2^n\int_{\Z_2^n}\|A_b\|_p^p\d\mu(b)$$
which is equivalent to (\ref{Inequality2}). Similarly, (\ref{Littlewood2}) becomes
$$\sum_{i=1}^{2^n}\left\|2^n\int_{\Z_2^n}k_i(b)A_b\d\mu(b)\right\|_p^q\le
2^{n}\left(2^n\int_{\Z_2^n}\|A_b\|_p^p\d\mu(b)\right)^{q/p}
$$
which is equivalent to (\ref{Inequalityp<2}). Finally (\ref{Littlewood3}) becomes
$$\sum_{i=1}^{2^n}\left\|2^n\int_{\Z_2^n}k_i(b)A_b\d\mu(b)\right\|_p^p\le
2^{n}\left(2^n\int_{\Z_2^n}\|A_b\|_p^q\d\mu(b)\right)^{q/p}
$$
which is equivalent to (\ref{Inequalityp>2})
\end{proof}

\begin{remark} A related result was given in \cite{Kato} (see also \cite[Theorem 3.3]{Maligranda}). It was proved
\begin{equation}\label{KatoUVP}
\left(\sum_{i=1}^{2^n}\left\|\sum_{j=1}^{2^n}\e_{ij}f_j\right\|_p^v\right)^{1/v}\le
2^{nc(u,v;p)}\left(\sum_{j=1}^{2^n}\|f_j\|_p^u\right)^{1/u},
\end{equation}
for various choices of $u$, $v$ and $p$, and for $f_i\in L^p$ -- the standard Lebesgue space. It was later generalized in \cite[Theorem 2.4]{Japonians} and \cite[Corollary 4.2]{Maligranda} for $f_i$ from Lebesgue Bochner space.

Corollary \ref{Little} is an expansion of (\ref{KatoUVP}) to $\Sc_p$ spaces for some choices of $u$, $v$. Namely, for $u=v=p\ge2$, for $u=p\le2$, $v=q=p/(p-1)$ and for $v=p\ge2$, $u=q=p/(p-1)$. The constants in theses cases match. We conjecture that similar can be done for other choices of $u$, $v$, $p$ (see Problem \ref{ProblemBoas}.)
\end{remark}

For $G=\T$ we get
\begin{corollary} For all $A_\theta\in L^p((0,2\pi);\Sc_p)$ we have
$$\sum_{k=-\infty}^{+\infty}\left\|\int_0^{2\pi}e^{-ik\theta}A_\theta\d\theta\right\|_p^p\le
    (2\pi)^{p-1}\int_0^{2\pi}\|A_\theta\|_p^p\d\theta\le\sum_{k=-\infty}^{+\infty}\alpha_k^{1-p/2}\left\|\int_0^{2\pi}e^{-ik\theta}A_\theta\d\theta\right\|_p^p,$$
for any sequence $\a_k>0$, $k\in\Z$ such that $\sum_{k=-\infty}^{+\infty}\a_k=1$.
$$\sum_{k=-\infty}^{+\infty}\left\|\int_0^{2\pi}e^{-ik\theta}A_\theta\d\theta\right\|_p^p\le
    2\pi\left(\int_0^{2\pi}\|A_\theta\|_p^q\d\theta\right)^{p/q},$$
for $p\ge2$ and $q=p/(p-1)$, whereas for $1\le p\le2$ and $q=p/(p-1)$, we have
$$\sum_{k=-\infty}^{+\infty}\left\|\int_0^{2\pi}e^{-ik\theta}A_\theta\d\theta\right\|_p^q\le
    2\pi\left(\int_0^{2\pi}\|A_\theta\|_p^p\d\theta\right)^{q/p},$$
\end{corollary}

\begin{proof} Consider $G=\T=\{z\in\C\;|\;|z|=1\}$. Its Pontryagin dual is $\hat G\cong\Z$, and the corresponding Haar measure is $\d\theta/2\pi$, i.e.\ the usual Lebesgue measure normalized by factor $2\pi$. Characters on $G$ are mappings $\theta\mapsto e^{k\theta}$ for all $k\in\Z$. Hence the result immediately follows from Theorems \ref{TheoremPP}, \ref{TheoremAlpha}, \ref{TheoremPQ} and \ref{TheoremPQ2}.
\end{proof}

\begin{remark}The preceding Corollary also estimates the Fourier coefficients for functions in Bochner spaces $L^p(\T;\Sc_p)$. More precisely, let $l^r(\Sc_p)$ denote the space of all sequences $B_k\in\Sc_p$, $k\in\Z$ such that $\sum_{k\in\Z}\|B_k\|_p^r<+\infty$ and let $\mathcal F$ stands for mapping which $A_\theta$ maps to $B_k=(1/2\pi)\int_0^{2\pi}e^{-ik\theta}A_\theta\d\theta$. Then preceding Corollary establishes norm estimates
$$\|\mathcal F\|_{L^p((0,2\pi);\Sc_p)\to l^p(\Sc_p)},\;\|\mathcal F\|_{L^q((0,2\pi);\Sc_p)\to l^p(\Sc_p)}\le1,\qquad p\ge2,\;q=p/(p-1),$$
$$\|\mathcal F\|_{L^p((0,2\pi);\Sc_p)\to l^q(\Sc_p)},\qquad p\le2,\;q=p/(p-2).$$
\end{remark}

\section{Problems}

We list some questions that naturally arises from results of this paper.

\begin{problem}\label{ProblemBoas} Does the inequalities of Boas-Koskela type (see \cite{Boas} and \cite{Koskela}), i.e.
$$\left(\sum_{k\in\hat G}\left\|\int_G\ovl{k(\theta)}A_\theta\d\mu(\theta)\right\|_p^r\right)^{1/r}\le
    \left(\int_G\|A_\theta\|_p^s\d\mu(\theta)\right)^{1/s}$$
for $s\le p\le r$ and $r/(r-1)\le s\le r$, hold?

If the answer is positive, it is likely that Kato inequality (\ref{KatoUVP}) might be extend to $\Sc_p$ classes for all choices of $u$, $v$ and $p$.
\end{problem}

\begin{problem} Can we prove that inequalities in Theorems \ref{TheoremPP}, \ref{TheoremPQ} and \ref{TheoremPQ2} are sharp? Once again, inequalities (\ref{KatoUVP}) are sharp for $f_i$ in Lebesgue spaces and Lebesque-Bochner spaces. It suggests, since constants match, that inequalities (\ref{Inequality2}), (\ref{Inequalityp<2}) and (\ref{Inequalityp>2}) are also sharp.
\end{problem}

\begin{problem} Is it possible to get convergence in $\luin\cdot\ruin$ norm in Theorem \ref{ParsevalThm}? We have strong convergence. If this convergence is in the norm of $J=B(H)$, this would lead to the conclusion that functions $k(\theta)I$ make a basis for Hilbert Module $L^2(G;B(H))$. See Remark \ref{HilbertModule}.
\end{problem}

\begin{problem} What can be done if $G$ is not assumed to be compact? Some classical results on Fourier transform may be useful. Namely for any $p>2$, there is $f\in L^p(\R^n)$ such that its Fourier transform is not a function, but tempered distribution. For $1\le p\le 2$, however, the Fourier transform is a bounded operator from $L^p(\R^n)$ to $L^q(\R^n)$, where $q=p/(p-1)$ with norm equals to $1$. For any locally compact group this is known as Haussdorf-Young inequality \cite[Proposition 4.28]{Folland}, and is usually proved by Riesz-Thorin theorem. This suggests, that (in the case where $G$ need not be compact), only inequality (\ref{Inequalityp<2}) might be generalized.
\end{problem}

\bibliographystyle{siam}
\bibliography{ContClarBibl}

\end{document}